\documentclass[a4paper,11pt]{amsart}
\usepackage[T2A]{fontenc}
\usepackage[cp1251]{inputenc}
\usepackage[english]{babel}
\usepackage{amssymb, amsmath, latexsym,amsthm}
\usepackage{amsfonts,mathtext,cite,enumerate,float}
\renewcommand{\epsilon}{\varepsilon}

\renewcommand{\Re}{\text{\rm Re\,}}
\renewcommand{\Im}{\text{\rm Im\,}}
\renewcommand{\le}{\leqslant}
\renewcommand{\ge}{\geqslant}

\newtheoremstyle{citedth}%
  {5pt}
  {5pt}
  {\itshape}
  {}
  {\bfseries}
  {.}
  {.3em}
  {\thmname{#1} \thmnumber{#2} \thmnote{\normalfont#3}}

\theoremstyle{citedth}
\newtheorem{theoremA}{Theorem}

\theoremstyle{theorem}
\newtheorem{theorem}{Theorem}

\theoremstyle{definition}
\newtheorem{example}{Example}
\newtheorem{remark}{Remark}

\usepackage{geometry}
\numberwithin{equation}{section}
\numberwithin{theorem}{section}
\numberwithin{example}{section}
\numberwithin{remark}{section}

\begin{document}
\author{Petr Chunaev}
\address{Universitat Aut\`{o}noma de Barcelona, Bellaterra (Barcelona) Spain and Vladimir State University, Vladimir, Russia} \email{chunayev@mail.ru, chunaev@mat.uab.cat}
\author{Vladimir Danchenko}
\address{Vladimir State University, Vladimir, Russia} \email{vdanch2012@yandex.ru}

\title[Quadrature formulas and inequalities for rational functions]{Quadrature formulas with variable nodes and Jackson-Nikolskii inequalities for rational functions}

\begin{abstract}
We obtain new parametric quadrature formulas with variable nodes for
integrals of complex rational functions over circles, segments of
the real axis and the real axis itself. Basing on these formulas we
derive $(q,p)$-inequalities of Jackson-Nikolskii type
 for various classes of rational functions, complex polynomials and their logarithmic derivatives (simple partial fractions). It is shown that our $(\infty,2)$- and $(\infty,4)$-inequalities are sharp in a number of main theorems. Our inequalities extend and refine several results obtained earlier by other authors.
\end{abstract}

\maketitle


\section{Introduction}

Sharp quadrature formulas for algebraic and trigonometric rational functions over circles and the real axis are well-known. Such formulas of interpolation type with fixed nodes (mostly at the zeros of Chebyshev-Markov fractions of the first and second kind) are constructed, for instance, in
\cite{Dirvuk,ROVBA2,ROVBA-DIRVUK,RUSAK1,OSIPENKO,MIN}.
Quadrature formulas contributed significantly to rational interpolation and approximation theory, in particular, to the study of approximating properties of special classes of rational fractions such as the (sine) cosine Chebyshev-Markov fractions.
 Moreover, they were actively used in various extremal problems, for example, for obtaining Markov-Bernstein, Szeg\H{o} and Nikolskii type inequalities for rational functions on circles and the real axis
  (see the history of this question and the references in \cite{[D3]}, \cite[Section 11]{Erdelyi} and \cite[Sections 1-2]{Baranov}).

Sharp quadrature formulas of another type were obtained recently in \cite{DanSem2016} (and partly earlier in~\cite{D2010}). The quadrature nodes in these papers are variable, depending on a certain parameter.
Under the corresponding change of this parameter, every node continuously varies over the whole domain of integration. This, in particular, allows us to choose one of the nodes at each prescribed point of the domain of integration.
This property is used in \cite{DanSem2016} in order to obtain the inequalities between various metrics (inequalities of Jackson-Nikolskii type) for rational functions and polynomials. We now cite one of such results that will be used below.

Let $r>0$ and $\gamma_r:=\{z:\,|z|=r\}$. Consider a finite set of pairwise distinct points on $\overline {\mathbb C}$ of the form
${\mathcal Z}_{\eta}\cup {\mathcal Z}_{\eta}^*$, ${\eta}\in
\mathbb N$, where
\begin{equation}
\label{Sets} {\mathcal
Z}_{\eta}:=\{{\xi}_1,\ldots,{\xi}_{{\eta}}\},\quad {\mathcal
Z}_{\eta}^*=\{{\xi}_1^{*},\ldots,{\xi}_{{\eta}}^{*}\}=
\{r^2/\overline{{\xi}_1},\ldots,r^2/\overline{{\xi}_{{\eta}}}\},\quad
{\xi}_1=0,\quad {\xi}_1^{*}=\infty,
\end{equation}
and the set ${\mathcal Z}_{\eta}$ lies in the disc $|z|<r$
(the set ${\mathcal Z}_{\eta}^*$ is symmetric to ${\mathcal
Z}_{\eta}$ with respect to the circle $\gamma_r$). Given $z\in \mathbb
C$, put
\begin{equation*}
\label{Def1} \hat{B}(z)=\frac{z}{r}\cdot\prod_{k=2}^{{\eta}}r
\frac{z-{\xi}_k}{r^2-z\overline{{\xi}_k}},\qquad
\hat{\mu}(z)=z\frac{\hat{B}'(z)}{\hat{B}(z)}.
\end{equation*}
It is easily seen that
\begin{equation*}
\label{Def2}
\hat{\mu}(\zeta)=\sum_{k=1}^{\eta}\frac{r^2-|{\xi}_k|^2}{|\zeta-{\xi}_k|^2}=1+\sum_{k=2}^{\eta}\frac{r^2-|{\xi}_k|^2}{|\zeta-{\xi}_k|^2},\qquad
\zeta\in\gamma_r.
\end{equation*}
Let $\zeta_k=\zeta_k(s,\varphi)$ be the roots of the equation
\begin{equation}
\label{NASECHKA} \hat{B}^s(\zeta)=e^{i\varphi}
 \end{equation}
with respect to $\zeta$ for some fixed $s\in \mathbb{N}$ and $\varphi\in \mathbb{R}$.
We call these roots \textit{notches}; they play an important role
in constructing quadrature formulas in \cite{DanSem2016}. It is not difficult to check that
for any  $s$ and $\varphi$ there are $s{\eta}$ different notches belonging to  $\gamma_r$. Indeed, the Blaschke
product $\hat{B}$ contains multipliers of the form
$$
v_k(\zeta):=r\frac{\zeta-{\xi}_k}{r^2-\zeta\overline{{\xi}_k}}=
\frac{r}{\zeta}\frac{\zeta-{\xi}_k}{r^2/\zeta-\overline{{\xi}_k}}=
\frac{r}{\zeta}\frac{\zeta-{\xi}_k}{\overline{\zeta}-\overline{{\xi}_k}}=
\frac{r}{\zeta}\frac{\zeta-{\xi}_k}{\overline{\zeta-{\xi}_k}}=e^{-i
\varphi(\zeta)}e^{2it(\zeta)}=e^{i(2t(\zeta)-\varphi(\zeta))},
$$
where $\zeta-{\xi}_k=r_1(\zeta)e^{it(\zeta)}$, $\zeta=re^{i
\varphi(\zeta)}$. Therefore, if the point $\zeta$ runs around the circle $\gamma_r$, then
 $\arg v_k(\zeta)=2t(\zeta)-\varphi(\zeta)$ has the increment $2\pi$. Note that the function $\arg v_k(\zeta)$
is monotone as a function of $\varphi$. If  the function $\arg v_k(\zeta)$ changed the direction at some point
$\zeta_0$, then the point
$v_k(\zeta)$ would change the direction (on the unit circle) and
the derivative $v_k'(\zeta_0)$ at $\zeta_0$  would vanish.
But the derivative of the linear fractional transformation does not vanish.
Thus, the continuous branch of the function $\arg{B}^s(\zeta)$ increases from
 $0$ up to $2\pi s{\eta}$ monotonically when $\zeta$ runs around the circle $\gamma_r$ in the positive direction. Consequently, $B^s(\zeta)$ runs around the unit circle continuously $s{\eta}$ times and hence the equation
(\ref{NASECHKA}) has $s{\eta}$ pairwise distinct roots.

The following statement holds.
\begin{theoremA}[\cite{DanSem2016}]
\label{DanSem_Th1}
Let $s\in \mathbb{N}$ and let $R$
be a rational function with poles on ${\mathcal Z}_{\eta}\cup {\mathcal Z}_{\eta}^*$; assume also that the poles different from ${\xi}_1=0$ and ${\xi}_1^{*}=\infty$ are of multiplicity at most  $s$ while the points ${\xi}_1$ and ${\xi}_1^{*}$ can
be poles of multiplicity at most $s-1$. Then for all $\varphi\in \mathbb{R}$
and the notches $\zeta_k=\zeta_k(s,\varphi)$, satisfying the equation
 $\hat{B}^{s}(\zeta)=e^{i\varphi}$, it holds that
\begin{equation}
\label{ThA-1} \int_{|\zeta|=r}R(\zeta)|d\zeta|=\frac{2\pi r}{s}
\sum_{k=1}^{s{\eta}}\frac{R(\zeta_k)}{\hat{\mu}(\zeta_k)}.
\end{equation}

Moreover, for any $m\in \mathbb{N}$ and the notches $\zeta_k=\zeta_k(2ms,\varphi)$, satisfying the equation
$\hat{B}^{2ms}(\zeta)=e^{i\varphi}$, it holds that
\begin{equation}
\label{ThA-2}
\|R\|_{L^{2m}(\gamma_r)}^{2m}=\int_{\gamma_r}|R(\zeta)|^{2m}|d\zeta|=\frac{\pi
r}{ms}\sum_{k=1}^{2ms{\eta}}\frac{|R(\zeta_k)|^{2m}}{\hat{\mu}(\zeta_k)},
\end{equation}
where it suffices to require that the sum of multiplicities of the poles
 ${\xi}_1=0$ and ${\xi}_1^{*}=\infty$ is at most $2ms-1$.
\end{theoremA}
A certain shortcoming of Theorem~\ref{DanSem_Th1} is that it takes into account not the multiplicity of each individual pole but the prescribed maximal multiplicity of the poles placed in $2\eta$ ``cells'' of the set ${\mathcal Z}_{\eta}\cup
{\mathcal Z}_{\eta}^*$. This fact significantly restricts the possibility to use Theorem~\ref{DanSem_Th1} in applications. We will prove a modified version of Theorem~\ref{DanSem_Th1}
--- Theorem~\ref{Th-1+}. This theorem already takes into account not the number of
``cells'' and their maximal multiplicity but the multiplicity of each pole separately. Because of that $s{\eta}$, the number of summands in the sum~(\ref{ThA-1}), can be exchanged for $1+\deg R\le s{\eta}$.

From Theorem~\ref{Th-1+} we will obtain the refinement of (\ref{ThA-2}) ---
the quadrature formula (\ref{Th1_Main_L_m}). Due to the variable nodes, this formula will imply $(q,p)$-inequalities of Jackson-Nikolskii type for rational functions on the circle $\gamma_r$. These inequalities refine the results obtained earlier in   \cite{Baranov} and \cite{DanSem2016}. What is more, our $(\infty,2)$-inequalities are sharp. As a particular case, we get sharp $(\infty,2)$-inequalities for complex polynomials, which were established earlier by various authors using other methods. We cite these results in the corresponding section.

We will also obtain analogues of modified Theorem~\ref{DanSem_Th1} and its applications
in the cases of the real axis, the real semiaxis and the segment $[-1,1]$.
In particular, we will obtain sharp $(\infty,2)$- and $(\infty,4)$-inequalities for logarithmic derivatives of complex polynomials on the real axis which refine several previous results.

\section{Modified quadrature formulas}

\subsection{The case of the circle $\gamma_r$}

\label{Section_Circle}
Basing on Theorem~\ref{DanSem_Th1}, we now prove the following main result about the quadrature formula in the case of the circle $\gamma_r$.
\begin{theorem} \label{Th-1+} Let  $r>0$ and let $R(z)$ be a rational function of degree~$n$ whose poles do not belong to the circle~$\gamma_r$. Set
\begin{equation*}
\label{mathcal_R}
{\mathcal R}(z):=R(z)\cdot\overline{R({r^2/\overline{z}})}.
\end{equation*}
We denote by $z_k$, $k=1,\ldots,\nu$, all pairwise distinct poles of the rational function~${\mathcal R}$ that lie in the disc~$|z|<r$ and by~$n_k$ their multiplicities. Set
\begin{equation}
\label{Blaschke_Product}
B(z)=\prod_{k=1}^{\nu}\left(r\frac{z-z_{k}}{r^2-z\overline{z_{k}}}
\right)^{n_k},
\qquad\mu(z)=\sum_{k=1}^{\nu}\frac{n_k(r^2-|z_k|^2)}{|z-z_k|^2},
\qquad\sum_{k=1}^{\nu} n_k=n.
\end{equation}
For $m\in {\mathbb N}$ and $\varphi\in \mathbb{R}$ let the notches~$\zeta_k(m,\varphi)$, $k=1,\ldots,mn+1$, be the roots of the equation~$\zeta B^m(\zeta)=re^{i\varphi}$. Then
\begin{equation*}
\label{Th1_Analogue_of_Th_A} \int_{\gamma_r}R(\zeta)|d\zeta|=2\pi
r\sum_{k=1}^{n+1}\frac{R({\zeta}_k)}
{\mu({\zeta}_k)+1},\qquad \zeta_k=\zeta_k(1,\varphi);
\end{equation*}
\begin{equation}
\label{Th1_Main_L_m}
\|R\|_{L^{2m}(\gamma_r)}^{2m}:=\int_{\gamma_r}|R(\zeta)|^{2m}|d\zeta|=2\pi
r\sum_{k=1}^{mn+1}\frac{|R({\zeta}_k)|^{2m}}{m{\mu}({\zeta}_k)+1},
\qquad \zeta_k=\zeta_k(m,\varphi).
\end{equation}
\end{theorem}
\begin{proof} To apply the formula~(\ref{ThA-1}), we use the procedure of
\textit{$\varepsilon$-separation of the poles} of the function $R$ that goes as follows. Set
$R_{\varepsilon,0}(z):=R(z/(1+\varepsilon z))$. Suppose that~$R$
has $\kappa$ pairwise distinct poles $t_k$ of multiplicity $m_k$,
$\sum_{k=1}^{\kappa} m_k=n=\deg R$ (possibly one of the poles is at infinity).
 Then the new rational function~$R_{\varepsilon,0}$ has the poles $v_k=t_k/(1-\varepsilon t_k)$ of multiplicity~$m_k$. Therefore
for a sufficiently small $\varepsilon>0$ the function
$R_{\varepsilon,0}(z)=P(z)/Q(z)$ has no poles on $\gamma_r$ and at infinity and furthermore $\deg R_{\varepsilon,0}=\deg R$.
Moreover, $R_{\varepsilon,0}(\zeta)\to R(\zeta)$ uniformly on $\gamma_r$ as $\varepsilon\to 0$.

We now exchange the polynomial
$$
 Q(z)=C\prod_{k=1}^{\kappa} (z-v_k)^{m_k}\qquad \hbox{for the polynomial}\qquad
Q_{\varepsilon}(z)=C\prod_{k=1}^{\kappa}\prod_{j=1}^{m_k}
(z-\xi_{k,j})
$$
with the simple zeros $\xi_{k,j}\notin \gamma_r$ (they are simple with respect to the whole double product) such that
$|v_k-\xi_{k,j}|\le \varepsilon$. In addition, we suppose that the set  $\cup_{k,j}\xi_{k,j}$ does not contain either the point
$z=0$ or pairs of points symmetric with respect to the circle $\gamma_r$.
Put $R_{\varepsilon,1}=P/Q_{\varepsilon}$. Obviously,
$R_{\varepsilon,1}(\zeta)\to R_{\varepsilon,0}(\zeta)\to R(\zeta)$ uniformly on $\gamma_r$
as $\varepsilon\to 0$.

Thus the function $R_{\varepsilon,1}$ has $n$ simple poles that differ from
$z=0$ and $z=\infty$ and have no pairs of poles symmetric with respect to  $\gamma_r$. We now symmetrize the set of these poles
with respect to $\gamma_r$ and add the points $0$ and $\infty$ to the set obtained. Then we get
the set $Z_{n+1}(\varepsilon)\cup
Z^*_{n+1}(\varepsilon)$ of the form (\ref{Sets}) with some
$\xi_k=\xi_k(\varepsilon)$ and $\eta=n+1$.

The functions $R_{\varepsilon,1}$ and $R_{\varepsilon,2}(z):=
R_{\varepsilon,1}(z)\cdot\overline{R_{\varepsilon,1}({r^2/\overline{z}})}$
have simple poles belonging to $Z_{n+1}(\varepsilon)\cup
Z^*_{n+1}(\varepsilon)$ and, consequently, satisfy the assumptions of Theorem~\ref{DanSem_Th1} with $s=1$ and $\eta=n+1$. So the formula (\ref{ThA-1}) holds for them:
\begin{equation*}
\label{Th1(1+++)}
\int_{\gamma_r}R_{\varepsilon,j}(\zeta)|d\zeta|=2\pi
r\sum_{k=1}^{n+1}\frac{R_{\varepsilon,j}(\zeta_k)}{{\mu_{\varepsilon}}(\zeta_k)},
\qquad {\mu}_{\varepsilon}(\zeta)=
1+\sum_{k=2}^{n+1}\frac{r^2-|\xi_k|^2} {|\zeta-\xi_{k}|^2},\qquad
\zeta\in\gamma_r,
\end{equation*}
where $j=1,2$ and the notches $\zeta_k=\zeta_k(\varphi)$
satisfy the equation
$\zeta{B}_{\varepsilon}(\zeta)=re^{i\varphi}$, where
\begin{equation*}
\label{Blaske_epsilon} {B}_{\varepsilon}(z)= \prod_{k=2}^{n+1}
r\frac{z-\xi_{k}} {r^2-z\overline{\xi_{k}}},\qquad
\xi_k=\xi_k(\varepsilon)\in Z_{n+1}(\varepsilon).
\end{equation*}
This implies the required statements of Theorem~\ref{Th-1+} for $m=1$ if
we take into account the uniform convergence
$R_{\varepsilon,1}(\zeta)\to R(\zeta)$,
$R_{\varepsilon,2}(\zeta)=|R_{\varepsilon,1}(\zeta)|^2\to
|R(\zeta)|^2={\mathcal R}(\zeta)$, $\mu_\varepsilon(\zeta)\to
\mu(\zeta)$, $B_\varepsilon(\zeta)\to B(\zeta)$ on $\gamma_r$ as $\varepsilon\to 0$.

Finally, considering the rational function $R^m(z)$ instead of $R(z)$, we obtain (\ref{Th1_Main_L_m}) for any $m\in
\mathbb{N}$.
\end{proof}

\subsection{The case of the real axes $\mathbb{R}$}
We need the following theorem from~\cite{DanSem2016} (see
\cite[Theorem 4]{DanSem2016} and its proof). Let the set ${\mathcal F}_\eta$ contain $2\eta$ pairwise distinct points not belonging to  ${\mathbb R}$ and have the form
\begin{equation*}
\label{set_F} {\mathcal F}_\eta=\{z_1,z_2,\ldots,z_{\eta}\}\cup
\{\overline{z_1},\overline{z_2},\ldots,\overline{z_{\eta}}\},\qquad
z_k\in {\mathbb C}^{+},\qquad z_k\ne\infty.
\end{equation*}
Put
\begin{equation}
\label{B-and-MU}
\hat{B}_{*}(x)=\prod_{k=1}^{\eta}\frac{x-z_k}{x-\overline z_k},
\qquad \hat{\mu}_{*}(x)=\frac{1}{2
i}\frac{\hat{B}'_{*}(x)}{\hat{B}_{*}(x)}=\sum_{k=1}^{\eta}\frac{{\rm
Im}\,z_k}{|x-\overline z_k|^2},\quad x\in {\mathbb R}.
\end{equation}
As in the case of the circle,  we use the notches, satisfying the equation (with respect to $x$)
$$
\hat{B}^s_{*}(x)=e^{i\varphi},\qquad s\in {\mathbb N},\qquad
\varphi\in {\mathbb R},
$$
in order to construct variable quadrature nodes on~${\mathbb R}$.
Since $|\hat{B}^s_{*}(x)|=1\Leftrightarrow x\in {\mathbb R}$ and
$\arg \hat{B}_{*}(x)$ is an increasing function on ${\mathbb R}$,
this equation has exactly $s\eta$ distinct real roots
 $x_k(\varphi)$ for each $\varphi$ (one of the roots may be at infinity). For a fixed $s$, if the parameter $\varphi$ increases continuously from $0$ up to $2\pi n s$, then each root $x_k(\varphi)$ runs around the circle
$\overline {\mathbb R}\subset \overline {\mathbb C}$ continuously.
\begin{theoremA}[\cite{DanSem2016}]
\label{Th_Int_formula_Real_axis_DanSem}
Let $R(z)$ be a proper rational fraction whose poles belong to  ${\mathcal F}_\eta$
and are of multiplicity at most $s$. Then for any $\varphi\in(0,2\pi)$ and the notches $x_k(\varphi)\in {\mathbb
R}$, satisfying the equation $\hat{B}_{*}^s(x)=e^{i\varphi}$,
we have
\begin{equation}
\label{Int_formula_Real_axis_DanSem}
 \int_{\mathbb R} R(x)
\,dx=\frac{\pi}{s}\sum_{k=1}^{s\eta}\frac{
R\left(x_k(\varphi)\right)} {\hat{\mu}_{*}(x_k(\varphi))}
\end{equation}
$($if the integral converges$)$. Moreover, for the notches   $x_k(\varphi)\in {\mathbb R}$, satisfying the equation
$\hat{B}_{*}^{2sm}(x)=e^{i\varphi}$, we have
\begin{equation}
\label{Int_formula_Real_axis_DanSem++}
\|R\|_{L^{2m}(\mathbb{R})}^{2m}=\frac{\pi}{2sm}\sum_{k=1}^{2sm\eta}
\frac{|R(x_k(\varphi))|^{2m}}{\hat{\mu}_{*}(x_k(\varphi))},\qquad
m\in \mathbb{N}.
\end{equation}
\end{theoremA}
Theorem~\ref{Th_Int_formula_Real_axis_DanSem} and the method of $\varepsilon$-separation of the poles give
\begin{theorem}
\label{Th_Real_axis}
Let  $R(z)$ be a proper rational fraction of degree $n$ whose poles do not belong to the real axis
$\mathbb{R}$. Set
\begin{equation*}
\label{mathcal_R_1} {\mathcal
R}_{*}(z):=R(z)\cdot\overline{R(\overline{z})}.
\end{equation*}
We denote by $z_k$, $k=1,\ldots,\nu$, all pairwise distinct poles of the rational function ${\mathcal R}_{*}$ which lie in the half-plane $\mathbb{C}^+$
 and by $n_k$ their multiplicities. Set
\begin{equation}
\label{Blaschke_Product_Real_Axis}
B_{*}(x)=\prod_{k=1}^{\nu}\left(\frac{x-z_{k}}{x-\overline{z}_{k}}
\right)^{n_k}, \qquad\mu_{*}(x)=\sum_{k=1}^{\nu}\frac{n_k{\rm
Im}\,z_k}{|x-\overline z_k|^2}, \qquad\sum_{k=1}^{\nu} n_k=n.
\end{equation}

Let $m\in {\mathbb N}$, $\varphi\in {\mathbb R}$ and let
the $($real$)$ notches $x_k=x_k(m,\varphi)$,
$k=1,\ldots,mn$, be the roots of the equation
$B_{*}^m(z)=e^{i\varphi}$. Then $($if the integrals converge$)$ we have
\begin{equation*}
\label{Th_Analogue_of_Th_B}
\int_{\mathbb{R}}R^m(x)dx=\frac{\pi}{m}
\sum_{k=1}^{mn}\frac{R^m(x_k)}{\mu_{*}(x_k)},
\end{equation*}
\begin{equation}
\label{Th1_Main_L_m_Real_axis}
\|R\|_{L^{2m}(\mathbb{R})}^{2m}:=\int_{\mathbb{R}}|R(x)|^{2m}
dx=\frac{\pi}{m}
 \sum_{k=1}^{mn}\frac{|R(x_k)|^{2m}}{{\mu_{*}}(x_k)},
\end{equation}
In the case $\varphi=0\,({\rm mod}\, 2\pi)$ one of the notches is at infinity.
\end{theorem}
It is sufficient to prove the statement for $m=1$ and then, considering $R^m(z)$ instead of $R(z)$
and $B_{*}^m$ and  $m\mu_{*}$ instead of $B_{*}$ and
$\mu_{*}$, correspondingly, we get the general form. For this we use the method of $\varepsilon$-separation of the poles analogous  to the one used in Section~\ref{Section_Circle} (the difference is that in the current situation the rational function  $R=P/Q$ has no poles at $z=0$ or $z=\infty$ (and therefore the change of variables is not necessary) and the separation is carried out so that there are no pairs of poles symmetric with respect to the real axis). Using the same notation and sequence of operations, we get the rational functions $R_{\varepsilon,1}=P/Q_{\varepsilon}$ and $R_{\varepsilon,2}(z):=
R_{\varepsilon,1}(z)\cdot\overline{R_{\varepsilon,1}(\overline{z})}$, whose poles are simple and belonging to the set of the form ${\mathcal
F}_{n}(\varepsilon)$. This set determines  ${\mu}_\varepsilon$ and ${B}_\varepsilon$
(as in (\ref{B-and-MU})). Besides the uniform convergence
$R_{\varepsilon,1}(x)\to R(x)$, $R_{\varepsilon,2}(x)\to
|R(x)|^2$, ${\mu}_\varepsilon(\zeta)\to {\mu}_{*}(\zeta)$,
${B}_\varepsilon(\zeta)\to {B}_{*}(\zeta)$ on ${\mathbb R}$, we also have to take into account that the corresponding integrals converge to each other.
It remains to use the formula~(\ref{Int_formula_Real_axis_DanSem}) with $s=1$ and $\eta=n$.

\subsection{The case of the real semiaxes $\mathbb{R}^+$}
Let  $R(z)$ be a proper rational fraction of degree~$n$ whose poles~$r_ke^{i\varphi_k}$ are of multiplicity $n_k$ so that $\sum_{k=1}^{\nu}
n_k=n$, $\varphi_k\in(0,2\pi)$ (not belonging to the real semiaxis~$\mathbb{R}^+$). Put
\begin{equation*}
\label{mathcal_R_1_semiaxis} {\mathcal
R}_{\mathbb{R}^+}(z):=R_0(z)\cdot\overline{R_0(\overline{z})},\qquad
R_0(z):=R(z^2).
\end{equation*}
Note that the poles of ${\mathcal R}_{\mathbb{R}^+}$ on the upper half-plane form the set
$$
\{z_k\}\cup\{-\overline{z_k}\}:=\{\sqrt{r_k}e^{i\varphi_k/2}\}
\cup\{-\sqrt{r_k}e^{-i\varphi_k/2}\}\subset {\mathbb C}^+,\qquad
k=\overline{1,\nu},
$$
where the poles~$z_k$, $-\overline{z_k}$ are of multiplicity~$n_k$,
$\sum_{k=1}^{\nu} n_k=n$. As in
$(\ref{Blaschke_Product_Real_Axis})$, we introduce the functions
$$
B_{*}(x)=\prod_{k=1}^{\nu}\left(\frac{x-z_{k}}{x-\overline{z}_{k}}
\frac {x+\overline{z}_{k}}{x+z_{k}} \right)^{n_k},\;
 \mu_*(x)=\sum_{k=1}^{\nu}\left(\frac{n_k\sqrt{r_k}
 \sin{\varphi_k}/{2}}
{|x-\sqrt{r_k}e^{i\varphi_k/2}|^2} +\frac{n_k\sqrt{r_k}
 \sin{\varphi_k}/{2}}
{|x+\sqrt{r_k}e^{-i\varphi_k/2}|^2}\right).
$$
\begin{theorem}
\label{Th_Real_semiaxis}
Let $m\in {\mathbb N}$ and let $R(z)$ satisfy the above-mentioned assumptions. We denote by $x_k=x_k(m,\varphi)$,
$k=1,\ldots,2mn$, the $($real$)$ notches, being the roots of the equation $B_{*}^m(x)=e^{i\varphi}$. Then
\begin{equation}
\label{Quadrature_Semiaxis}
\int_{\mathbb{R}^+}\frac{R(x)}{\sqrt{x}}dx=\pi
\sum_{k=1}^{2n}\frac{R(x_k^2)}{\mu_*(x_k)},\quad
\|R\|_{L^{2m}(\mathbb{R}^+;\frac{1}{\sqrt{x}})}^{2m}:=
\int_{\mathbb{R}^+}\frac{|R(x)|^{2m}}{\sqrt{x}} dx=\frac{\pi}{m}
 \sum_{k=1}^{2mn}\frac{|R(x_k^2)|^{2m}}{{\mu_{*}}(x_k)}.
\end{equation}
In the case $\varphi=0\,({\rm mod}\, 2\pi)$ one of the notches is at
infinity.
\end{theorem}
\begin{proof} Let $m=1$.
We now substitute $x=\tilde{x}^2$ in the integrals in~(\ref{Quadrature_Semiaxis}). Then, if
$R_0(\tilde{x})=R(\tilde{x}^2)$,
\begin{equation*}
\label{ZAMENA-R+} \int_{\mathbb{R}^+}\frac{R(x)}{\sqrt{x}}dx=
\int_{\mathbb{R}}R_0(\tilde{x})d\tilde{x},\qquad
\int_{\mathbb{R}^+}\frac{|R(x)|^2}{\sqrt{x}}dx=
\int_{\mathbb{R}}|R_0(\tilde{x})|^2d\tilde{x}.
\end{equation*}
Thus we get the required formulas~(\ref{Quadrature_Semiaxis}) by applying Theorem~\ref{Th_Real_axis} with the exchange of $R$ for $R_0$ in the integrals over~$\mathbb{R}$. The general case $m\in {\mathbb N}$ follows if we exchange $R$ for $R^m$.
\end{proof}

\begin{remark}  One can consider the weight $\sqrt{x}$ instead of $1/\sqrt{x}$. Indeed, in this case, for example, for $m=1$ by substituting $x=\tilde{x}^2$ we get
$$
\int_{\mathbb{R}^+}R(x)\sqrt{x}dx=\int_{\mathbb{R}}R_1(\tilde{x})d\tilde{x},\qquad
R_1(\tilde{x})=\tilde{x}^2R_0(\tilde{x})=\tilde{x}^2R(\tilde{x}^2).
$$
Furthermore, as in Theorem~\ref{Th_Real_semiaxis}, applying
Theorem~\ref{Th_Real_axis} to the function $R_1$ (of degree $2n$)
gives the following formulas  (if the integrals converge):
$$
\int_{\mathbb{R}^+}R(x)\sqrt{x}dx=\pi
\sum_{k=1}^{2n}\frac{x_k^2R(x_k^2)}{\mu_*(x_k)},\quad
\int_{\mathbb{R}^+}|R(x)|^2\sqrt{x}dx=\pi
\sum_{k=1}^{2n}\frac{x_k^2|R(x_k^2)|^2}{\mu_*(x_k)},
$$
where $x_k$ are the same notches as in (\ref{Quadrature_Semiaxis}) with $m=1$.
\end{remark}

\subsection{The case of the interval $[-1,1]$}
In this section we prove an analogue of \cite[Theorem 3]{DanSem2016} with the same refinements as we mentioned in the cases of the circle and the real axis. The proof is easily reduced to the case of the circle by the Zhukovsky substitution and hence we only give brief arguments. Set
\begin{equation*}
\label{weight_segment} I:=[-1,1],\qquad
\omega(x):=\frac{1}{\sqrt{1-x^2}}.
\end{equation*}
\begin{theorem}
\label{Segment_Th_Real_axis}
Let $R(z)$ be a rational function of degree~$n$ whose poles do not belong to the segment~$I$. Put
\begin{equation*}
\label{Segment_mathcal_R_1} {\mathcal
R}(z):=R_1(z)\cdot\overline{R_1({1/\overline{z}})},\quad
R_1(z):=R(\tfrac{1}{2}(z+1/z)).
\end{equation*}
We denote by $z_k$, $k=1,\ldots,\nu$, all pairwise distinct poles of the function ${\mathcal R}$, which lie in the disc~$|z|<1$, and by $n_k$ their multiplicities. Set
\begin{equation*}
\label{Segment_Blaschke_Product}
B_0(z)=\prod_{k=1}^{\nu}\left(\frac{z-z_{k}}{1-z\overline{z_{k}}}
\right)^{n_k},
\qquad\mu_0(z)=\sum_{k=1}^{\nu}\frac{n_k(1-|z_k|^2)}{|z-z_k|^2},
\qquad\sum_{k=1}^{\nu} n_k=2n.
\end{equation*}

Then for any $\varphi\in \mathbb{R}$ we have
\begin{equation}
\label{Segment_Int_formula}
 \int_I R(x)\omega(x)
\,dx=\pi\sum_{k=1}^{2n+1}\frac{ R\left(x_k(\varphi)\right)}
{\mu_{0}(\zeta_k(\varphi))+1},\qquad
x_k=\frac{1}{2}\left(\zeta_k+\frac{1}{\zeta_k}\right)\in [-1,1],
\end{equation}
where the notches $\zeta_k=\zeta_k(\varphi)$ are the roots of the equation~$\zeta
B_{0}(\zeta)=e^{i\varphi}$;
\begin{equation}
\label{Segment_L_m} \|R\|_{L^{2m}(I;\,\omega)}^{2m}:=\int_I
|R(x)|^{2m}\omega(x) \,dx=\pi\sum_{k=1}^{2nm+1}
\frac{|R(x_k(m,\varphi))|^{2m}}{m\mu_{0}(\zeta_k(m,\varphi))+1},\qquad
m\in {\mathbb N},
\end{equation}
where all $x_k=x_k(m,\varphi)$ depend on the notches $\zeta_k(m,\varphi)$ as in $(\ref{Segment_Int_formula})$ while $\zeta_k(m,\varphi)$ are the roots of the equation $\zeta B_{0}^{m}(\zeta)=e^{i\varphi}$.
 \end{theorem}
\begin{proof}
We derive the formulas (\ref{Segment_Int_formula}) and (\ref{Segment_L_m})
in the same way as in Theorem~\ref{Th-1+} taking into account that the substitution $x=\tfrac{1}{2}(\zeta+1/\zeta)$ gives (see, for example, \cite{DanSem2016}):
\begin{equation*}
\label{ZAMENA+++} \int_I R(x)\omega(x) \,dx=
\frac{1}{2}\int_{\gamma_1}R_1(\zeta)|d\zeta|,\qquad
\|R\|_{L^{2m}(I;\,\omega)}^{2m}=
\frac{1}{2}\|R_1\|_{L^{2m}(\gamma_1)}^{2m}.
\end{equation*}
\end{proof}

\section{Quadrature formulas and inequalities in the case of the circle $\gamma_r$}

\subsection{Pointwise estimates of Jackson-Nikolskii type for rational functions}
The following result is a corollary of the quadrature formula~(\ref{Th1_Main_L_m}).
\begin{theorem}
\label{Th_Nikolski_Circle_poinwise} Let $m\in \mathbb{N}$ and let $R(z)$
be a rational function of degree~$n$ whose poles do not belong to the circle~$\gamma_r$. Then the following inequality holds:
\begin{equation}
\label{Main_Inequality_Circle}
 \frac{|R(\zeta)|^{2m}}{m\mu(\zeta)+1}\le
\frac{1}{2\pi r}\|R\|_{L^{2m}(\gamma_r)}^{2m}\qquad \forall \zeta\in
\gamma_r,
\end{equation}
where $\mu$ is defined in $(\ref{Blaschke_Product})$. This inequality is sharp for  $m=1$: for any $\zeta_0\in\gamma_r$ there exists a rational function $R^{\star}$ such that $(\ref{Main_Inequality_Circle})$ becomes an equality if
$\zeta=\zeta_0$.
\end{theorem}
\begin{proof}
Each point  $\zeta_k(m,\varphi)$ in the formula~(\ref{Th1_Main_L_m})
runs around the circle~$\gamma_r$ continuously if the parameter $\varphi$ increases from $0$ up to  $2\pi (mn+1)$ continuously.
Therefore for any given $\zeta\in \gamma_r$ one can choose $\varphi$ so that $\zeta_1(m,\varphi)=\zeta$. So if we keep just the first summand in the sum (\ref{Th1_Main_L_m}) that contains only positive terms, then we get~(\ref{Main_Inequality_Circle}).
\end{proof}
The following example shows the sharpness of (\ref{Main_Inequality_Circle}) for $m=1$.
\begin{example}
 \label{example1}
Let
$$
R^{\star}(z)=R^{\star}(n,\varphi;z)=\frac{zB(z)-re^{i
\varphi}}{z-\zeta_1}=C\frac{\prod_{k=2}^{n+1}(z-\zeta_k)}{\prod_{k}(r^2-z\overline{z_{k}})^{n_k}},\qquad
\zeta_k=\zeta_k(1,\varphi),
$$
where $B$ is the Blaschke product from (\ref{Blaschke_Product}) and
$\varphi\in {\mathbb R}$. The function $R^{\star}$ is of degree $n$ and satisfies the assumptions of Theorem~\ref{Th-1+}.
 Consequently, the equality (\ref{Th1_Main_L_m}) holds for it. From the other side, for all
 the notches different from $\zeta_1$ we have
$R^{\star}(\zeta_k)=0$, where $k=2,\ldots,n+1$. By this reason the inequality (\ref{Main_Inequality_Circle})
becomes an equality at the point
$\zeta=\zeta_1$. What is more, the function
$|R^{\star}(\zeta)|^{2}/(\mu(\zeta)+1)$ reaches its maximal value on $\gamma_r$ at this point. Hence if in addition $\varphi$ is chosen so that $\mu(\zeta_1(1,\varphi))=\|\mu\|_{L^{\infty}(\gamma_r)}$, then
\begin{equation}
\label{EXTREMA-R}
 \|R^{\star}\|_{L^{2}(\gamma_r)}^{2}= 2\pi
r
\frac{\|R^{\star}\|_{L^{\infty}(\gamma_r)}^{2}}
 {\|\mu\|_{L^{\infty}(\gamma_r)}+1}.
\end{equation}

Another example of an extremal rational function for (\ref{Main_Inequality_Circle}) was constructed in~\cite[Section~3.1.1]{DanSem2016}.
\end{example}

\begin{remark}
Note that the estimate $(\ref{Main_Inequality_Circle})$ refines a similar result from~\cite[Theorem 5]{DanSem2016}, where the weight $\mu$
(see (\ref{Blaschke_Product})) contains the maximal multiplicity $\max\{n_k\}$ instead of the individual multiplicities~$n_k$.
\end{remark}
\begin{remark}
For each $\zeta\in\gamma_r$ the inequality (\ref{Main_Inequality_Circle}) obviously yields the following alternative (cf.
Theorem~\ref{Theorem_Alternative_Real_axis}):
\begin{equation*}
\label{Theorem_Alternative_Circle_1}
 |R(\zeta)| <\frac{m\mu(\zeta)+1}{2\pi r}\qquad
\text{or} \qquad |R(\zeta)|^{2m-1}\le
 \|R\|_{L^{2m}(\gamma_r)}^{2m}.
\end{equation*}
In particular, for $m=1$ we get
$$
|R(\zeta)|\le \max\left\{\frac{\mu(\zeta)+1}{2\pi
r},\,\|R\|_{L^{2}(\gamma_r)}^{2}\right\}.
$$
\end{remark}

\subsection{$(q,p)$-inequalities of Jackson-Nikolskii type for rational functions}
\label{L-p-q+++}

Another corollary of Theorem~\ref{Th_Nikolski_Circle_poinwise} is
\begin{theorem}
\label{theorem3.2}
Under the assumptions of Theorem~$\ref{Th_Nikolski_Circle_poinwise}$, the following
$(q,p)$-inequality of Jackson-Nikolskii type holds:
\begin{equation}
\label{Nikolsky_ineq_circle} \|R\|_{L^q(\gamma_r)}\le
\left(\frac{m_p\,\|\mu\|_{L^\infty(\gamma_r)}+1}{2\pi
r}\right)^{\frac{1}{p}-\frac{1}{q}}\|R\|_{L^p(\gamma_r)}, \qquad
0<p<q\le\infty,
\end{equation}
where  $m_p\in\mathbb{N}\cap[\tfrac{p}{2};1+\tfrac{p}{2})$. Moreover, if all the poles of $R$ do not belong to the annulus
$$
A_\delta:=\left\{z:\delta r <|z|< \delta^{-1}r\right\},\qquad
\delta\in (0,1),
$$
then for $0<p<q\le\infty$ we have
\begin{equation}
\label{Nikolsky_ineq_circle_delta} \|R\|_{L^q(\gamma_r)}\le
\left(\frac{1}{2\pi
r}\right)^{\frac{1}{p}-\frac{1}{q}}
\left(m_p\,n\cdot\frac{1+\delta}{1-\delta}+1
\right)^{\frac{1}{p}-\frac{1}{q}}\|R\|_{L^p(\gamma_r)}.
\end{equation}
The inequalities $(\ref{Nikolsky_ineq_circle})$ and
$(\ref{Nikolsky_ineq_circle_delta})$ are sharp for $(q,p)=(\infty,2)$.
\end{theorem}
\begin{proof} Choose $\zeta$ in (\ref{Main_Inequality_Circle})
so that $|R(\zeta)|=\|R\|_{L^\infty(\gamma_r)}$. Then, exchanging
$\mu(\zeta)$ for $\|\mu\|_{L^\infty(\gamma_r)}$ in
(\ref{Main_Inequality_Circle}), we get
$$
\|R\|_{L^\infty(\gamma_r)}^{2m}\le
C_m\|R\|_{L^{2m}(\gamma_r)}^{2m},\qquad
C_m:=\frac{m\|\mu\|_{L^\infty(\gamma_r)}+1}{2\pi r}.
$$
Furthermore, we now use the well-known trick to obtain the $(q,p)$-inequality (see, e.g. \cite[Section~4.9.2]{Timan} and
\cite[Section 4]{DeVore}). For $2(m-1)< p\le 2m$
($m=m_p\in\mathbb{N}\cap[\tfrac{1}{2}p;\tfrac{1}{2}p+1)$) we have
$$
\|R\|^{2m}_{L^{2m}({\gamma_r})}=\int_{\gamma_r} |R|^{2m-p}|R|^{p}\le
\|R\|^{2m-p}_{L^{\infty}({\gamma_r})} \|R\|^{p}_{L^{p}({\gamma_r})},
$$
which, together with the latter estimate, implies that
$$
\|R\|^{2m}_{L^{\infty}({\gamma_r})}\le C_m
\|R\|^{2m-p}_{L^{\infty}({\gamma_r})}
\|R\|^{p}_{L^{p}({\gamma_r})}\quad\Rightarrow\quad
\|R\|_{L^{\infty}({\gamma_r})}\le
C_m^{1/p}\|R\|_{L^{p}({\gamma_r})}.
$$
From this for any $q\ge p$ we conclude that
$$
\|R\|_{L^q({\gamma_r})}^q=\int_{\gamma_r}
|R|^{q-p}|R|^p\le\|R\|_{L^\infty({\gamma_r})}^{q-p}\|R\|_{L^p({\gamma_r})}^{p}
$$
$$
\le
\left(C_m^{1/p}\|R\|_{L^{p}({\gamma_r})}\right)^{q-p}\|R\|_{L^p({\gamma_r})}^{p}=C_m^{(q-p)/p}\|R\|_{L^p({\gamma_r})}^{q}.
$$
Raising each side to the power $1/q$ leads to the inequality~(\ref{Nikolsky_ineq_circle}). This inequality implies~(\ref{Nikolsky_ineq_circle_delta}) if we take into account the estimate for the weight~(\ref{Blaschke_Product}):
$$
\mu(\zeta)=\sum_{k=1}^{\nu}\frac{n_k(r^2-|z_k|^2)}{|\zeta-z_k|^2}
\le \sum_{k=1}^\nu n_k\frac{r^2-(\delta r)^2}{(r-\delta r)^2}
 = n\; \frac{1+\delta}{1-\delta},\quad z_k\notin A_\delta,
 \quad |\zeta|=r.
 $$

We now give an example to show the sharpness of (\ref{Nikolsky_ineq_circle_delta}) in the case $(q,p)=(\infty,2)$. The sharpness of~$(\ref{Nikolsky_ineq_circle})$ was discussed in the previous example, see  (\ref{EXTREMA-R}).
\begin{example}
 \label{example2}
Let us prove the sharpness of $(\ref{Nikolsky_ineq_circle_delta})$ as follows. Let $B$
have the single zero $z_{1}=\delta r$ of multiplicity $n$ and let the quadrature nodes $\zeta_k$ correspond to $\varphi=0$.
Then one of the nodes, say $\zeta_1$, equals~$r$. Consequently,
$$
B(z)=\left(\frac{z-\delta r}{r-z\delta}\right)^{n},\quad
\mu(\zeta)=nr^2\frac{1-\delta^2}{|\zeta-\delta r|^2},\quad
R^{\star}(z)=\frac{zB(z)-r}{z-r}.
$$
If $\zeta=r$,  $m=1$ and $R=R^{\star}$, then the left hand side of~(\ref{Main_Inequality_Circle}) reaches its maximal value and we get the equality (see the discussion in Example~\ref{example1}).
Moreover, $\mu(\zeta)+1$ reaches its maximal value
$n\cdot\frac{1+\delta}{1-\delta}+1$ at the same point and hence we get the equality in~(\ref{Nikolsky_ineq_circle_delta}) for $(q,p)=(\infty,2)$.
\end{example}
\end{proof}
\begin{remark}
The estimate (\ref{Nikolsky_ineq_circle_delta}) refines the following result from~\cite{DanSem2016}:
$$
\|R\|_{L_{\infty}(\gamma_r)} \le \sqrt{\frac{N s}{2\pi
r}\,\frac{1+\delta}{1-\delta}}\;\|R\|_{L_2(\gamma_r)}, \qquad
z_k\notin A_{\delta},
$$
where $s$ is the maximal multiplicity of the poles of $R$ and $N=2\eta$ is the number of its pairwise distinct poles (see the definitions in Theorem~\ref{DanSem_Th1}).  In addition, $sN\ge n$ and we have the equality only in the case when all the poles are of the same multiplicity~$s$. Later on, the $(q,p)$-version of this inequality with the corresponding constant was obtained in~\cite{Platov}.

In the recent paper \cite[Theorem 2.4]{Baranov} other methods are used to establish inequalities of the form  (\ref{Nikolsky_ineq_circle_delta}) and the following constant is obtained
before $\|R\|_{L^p(\gamma_r)}$:
\begin{equation}
\label{BARANOV} \left(\frac{1}{2\pi
r}\right)^{\frac{1}{p}-\frac{1}{q}}
  \left(m_p\,n+1\right)^{\frac{1}{p}-\frac{1}{q}}
  \left(\frac{1+\delta}{1-\delta}\right)^{\frac{1}{p}-\frac{1}{q}},
  \qquad m_p\in\mathbb{N}\cap[\tfrac{p}{2};1+\tfrac{p}{2}).
\end{equation}
As for sharpness, it is shown in \cite{Baranov} that there exists a constant  $c(p,q)>0$ such that for all  $n\ge 2$ and $\delta\in
(0,1)$ the following inequality holds:
$$
\sup \frac{\|R\|_{L^q(\gamma_r)}}{\|R\|_{L^p(\gamma_r)}}\ge c(p,q)
\left(\frac{n}{1-\delta}\right)^{\frac{1}{p}-\frac{1}{q}},
$$
where the $\sup$ is taking over all rational functions of degree~$n$ with poles outside the annulus $A_{\delta}$. What is more, there is an example in \cite{Baranov} showing the \textit{asymptotic} sharpness of the constant~(\ref{BARANOV}) when
$(q,p)=(\infty,2)$ and $n\to \infty$.

We see that the constant in (\ref{Nikolsky_ineq_circle_delta}) is better than
in (\ref{BARANOV}). Moreover, our example shows that our inequality is sharp for $(q,p)=(\infty,2)$.
\end{remark}

\subsection{Jackson-Nikolskii type inequalities for polynomials}
The study of inequalities between various metrics (inequalities of Jackson-Nikolskii type) for algebraic and trigonometric polynomials was initiated by Jackson~\cite{Jackson} and
Nikolskii~\cite{Nikolskiy}. Nowadays there are many results on this topic  (see the history of this question and numerous references in  \cite{Arestov-Deykalova} and
\cite{GanzburgTikh}). In this section we obtain a number of known inequalities for polynomials as a particular case of Theorem~\ref{Nikolsky_ineq_circle}.

Consider a rational function $R_{n_1,n_2}(z)=\sum_{k=-n_1}^{n_2} \sigma_kz^k$ $(n_1,n_2\ge
0,\;\sigma_k\in\mathbb C)$. If $n_1=0$, then we get an algebraic polynomial $P_{n_2}$. Let $l=n_1+n_2$ be the sum of multiplicities of the poles of the function $R_{n_1,n_2}$ at the points $z=0$ and $z=\infty$.
As $\delta\to
0$, the inequality (\ref{Nikolsky_ineq_circle}) (and
(\ref{Nikolsky_ineq_circle_delta}) or (\ref{BARANOV})) immediately yields the inequality
\begin{equation}
\label{Nikolsky_ineq_polynomials}
\|R_{n_1,n_2}\|_{L^q(\gamma_r)}\le \left(\frac{m_p\,l+1}{2\pi
r}\right)^{\frac{1}{p}-\frac{1}{q}}\|R_{n_1,n_2}\|_{L^p(\gamma_r)}, \qquad l=n_1+n_2.
\end{equation}
For $(q,p)=(\infty,2)$ ($m_2=1$) and $n_1=0$ this inequality is sharp and the extremal polynomial is
$P_{n}(z)=(z^{n+1}-r^{n+1})/(z-r)$ (see, for example,
\cite[Inequality (29)]{DanSem2016}). The estimate~(\ref{Nikolsky_ineq_polynomials}) refines \cite[Inequality~(29)]{DanSem2016} for $m_p\ge 2$.

The problem of the sharp constant in inequalities of the form~(\ref{Nikolsky_ineq_polynomials}) for algebraic polynomials,
i.e. for $n_1=0$, seems to be unsolved for all pairs $(q,p)$ except $(\infty,2)$. Progress in this direction is recently made in~\cite{Levin-Lubinsky} and \cite{GanzburgTikh}, where the above-mentioned problem is reduced to a certain extremal problem for entire functions on the real axis~$\mathbb{R}$.

It is easily seen that for $r=1$ one can go from the rational functions $R_{n_1,n_2}$ to the trigonometric polynomials with complex coefficients and back using the change of variables  $z=e^{it}$.
Let us consider several examples of using the inequality~(\ref{Nikolsky_ineq_polynomials}) in this context.

Let  $T_{n_1,n_2}(t):=R_{n_1,n_2}(e^{it})$ be a trigonometric polynomial. Then (\ref{Nikolsky_ineq_polynomials}) for $r=1$
gives the inequality
\begin{equation}
\label{Nikolsky_ineq_trig_polynomials}
\|T_{n_1,n_2}\|_{L^q[0,2\pi]}\le \left(\frac{m_p\,
l+1}{2\pi}\right)^{\frac{1}{p}-\frac{1}{q}}
 \|T_{n_1,n_2}\|_{L^p[0,2\pi]},\qquad l=n_1+n_2.
\end{equation}
In particular, given a trigonometric polynomial with complex coefficients of the standard form $T_n:=T_{n,n}$, we have
\begin{equation}
\label{Nikolsky_ineq_trig_polynomials+} \|T_{n}\|_{L^q[0,2\pi]}\le A\, n^{\frac{1}{p}-\frac{1}{q}}
\|T_{n}\|_{L^p[0,2\pi]},\qquad A:=\left(\frac{2m_p+1/n}{2\pi}\right)^{\frac{1}{p}-\frac{1}{q}}.
\end{equation}
The inequality (\ref{Nikolsky_ineq_trig_polynomials+}) is well-known; other methods to prove it can be found in the monographs \cite[Section~4.9.2]{Timan} and \cite[Section 4,
Theorem~2.6]{DeVore}. It is sharp for $(q,p)=(\infty,2)$, $m_2=1$. Absolute upper estimates of $A$ (independent of $p$,
$q$ and $n$) were considered in a number of papers. The best estimate $A\le 1.08$  known to us was obtained in  \cite{GanzburgTikh}. Other inequalities of the form (\ref{Nikolsky_ineq_trig_polynomials+}),
including the ones with more precise constants for some $(q,p)$, were obtained
by different methods, for example, in~\cite[Corollary
6.4]{Ditzian-Tikhonov}, \cite[Remark 6.3]{Ditzian-Prymak-2010} and
\cite[Corollary 4]{Mamedkhanov}.

There exist estimates of the constant in~(\ref{Nikolsky_ineq_trig_polynomials+}) with $(q,p)=(\infty,1)$
(see the results and references on this topic in~\cite{Gorbachev}). The classical estimates belong to Stechkin and Taykov who showed that the precise multiplier before $n$ is of the form $C/\pi+o(1)$ with some $C\in(0.53;0.59)$ as $n\to \infty$. Gorbachev \cite{Gorbachev} obtained some refinements of these results.
It is interesting to note that \cite[Section 3.1.2]{DanSem2016} contains the following inequality for
\textit{real non-negative} trigonometric polynomials:
$$
\|T_n\|_{L^\infty[0,2\pi]}\le  \frac{n+1}{2\pi}\,
\|T_n\|_{L^1[0,2\pi]}.
$$
One can use the estimate (\ref{Nikolsky_ineq_trig_polynomials}) even in the case of trigonometric polynomials of a special form. For example,
\begin{equation}
\label{Nikolsky_ineq_trig_polynomials+++}
\|T_{n+1,n}\|_{L^q[0,2\pi]}\le
\left(\frac{2m_p\,n+m_p+1}{2\pi}\right)^{\frac{1}{p}-\frac{1}{q}}
 \|T_{n+1,n}\|_{L^p[0,2\pi]}.
\end{equation}
For $(q,p)=(\infty,2)$, as shown in~\cite[Section 3.1.2]{DanSem2016},
the inequality (\ref{Nikolsky_ineq_trig_polynomials+++}) is sharp and the extremal polynomial is
$$
T^{\star}_{n+1,n}(t)=1+2\sum_{k=1}^n\cos kt+e^{it(n+1)}.
$$

To finish, we want to mention that there are many papers where the inequalities of Jackson-Nikolskii type for polynomials are extended to various classes of weights and bounded domains. The overview of these results and the corresponding references can be found, for example, in \cite{Ditzian-Prymak-2010,Ditzian-Prymak,Mamedkhanov,Arestov-Deykalova}. We also cite some of these results in the forthcoming sections of this paper.

\section{Quadrature formulas and inequalities in the case of the real axis $\mathbb{R}$}

\subsection{Pointwise inequalities of Jackson-Nikolskii type for rational functions}

The following result is a corollary of the quadrature formula~(\ref{Th1_Main_L_m_Real_axis}).
\begin{theorem}
\label{Th_Nikolski_Real_axis_poinwise}
Let $R(z)$ be a proper rational fraction of degree $n$ whose poles do not belong to the real axis  $\mathbb{R}$. Then the following inequality holds:
\begin{equation}
\label{Main_Inequality_Real_axis}
 \frac{|R(x)|^{2m}}{\mu_{*}(x)}\le
\frac{m}{\pi}\|R\|_{L^{2m}(\mathbb{R})}^{2m}\qquad \forall x\in
\mathbb{R},\qquad m\in \mathbb{N},
\end{equation}
where $\mu_{*}$ is defined in $(\ref{Blaschke_Product_Real_Axis})$.
This inequality is sharp for $m=1$ and $m=2$:  there exist a rational function $R^{\star}$ and $x=x_0$ such that
$(\ref{Main_Inequality_Real_axis})$ becomes an equality.
\end{theorem}
\begin{proof}
Each point  $x_k(m,\varphi)$ in the formula~(\ref{Th1_Main_L_m_Real_axis}) runs along $\mathbb{R}$ continuously if the parameter $\varphi$ changes continuously. Therefore, for any $x\in \mathbb{R}$ we can choose $\varphi$ so that $x_{1}(m,\varphi)=x$.
So if we keep just the first summand in the sum (\ref{Th1_Main_L_m_Real_axis}) that contains only positive terms, then we get~(\ref{Main_Inequality_Real_axis}).
\end{proof}

The sharpness of~(\ref{Main_Inequality_Real_axis}) for $m=1$ and $m=2$ is confirmed by the following examples.
\begin{example}
 \label{example3}
Argument are analogous to those in Example~\ref{example1}. Let
$$
R^{\star}(z)=R^{\star}(n,\varphi;z)=\frac{B_{*}(x)-e^{i
\varphi}}{x-x_1}=C\frac{\prod_{k=2}^{n}(x-x_k)}{\prod_{k=1}^\nu(x-\overline{z_{k}})^{n_k}},\quad
x_k=x_k(1,\varphi),\quad \varphi\in(0,2\pi),
$$
where $B_{*}$ is the Blaschke product from~(\ref{Blaschke_Product_Real_Axis}). The function $R^{\star}$ is of degree
 $n$ and satisfies the assumptions of Theorem~\ref{Th_Real_axis}, hence
the equality~(\ref{Th1_Main_L_m_Real_axis}) with $m=1$ holds for it. Furthermore, by the same calculations as in Example~\ref{example1}, for some $\varphi$ we get
\begin{equation*}
\label{EXTREMA-R++}
 \|R^{\star}\|_{L^{2}(\mathbb{R})}^{2}= \pi
\frac{\|R^{\star}\|_{L^{\infty}(\mathbb{R})}^{2}}
 {\|\mu_{*}\|_{L^{\infty}(\mathbb{R})}}.
\end{equation*}
\end{example}
\begin{example}
\label{example4} Given some $p>1$, consider the rational function
\begin{equation}
\label{rho_wave} \rho(z)={\rho}(p;z):=\frac{1}{z-y_0i}, \qquad
y_0=y_0(p):=\left(\frac{\pi
 2^{2-p}}{(p-1)\,\textrm{B}(p/2, p/2)}\right)^{\frac{1}{p-1}}>0,
\end{equation}
where $\textrm{B}$ is the Euler beta function and $y_0(p)\to 1$ as
$p\to\infty$. One can easily check (see, for example, \cite[Section 2]{Borodin2007}) that $\|{\rho}\|_{L^p(\mathbb{R})}=1$. Taking into account the obvious equalities
$$
{\rho}(0)=\|{\rho}\|_{L^\infty(\mathbb{R})}=1/y_0,\qquad
\mu_{*}(0)=\|\mu_{*}\|_{L^\infty(\mathbb{R})} =1/y_0,
$$
we rewrite the inequality (\ref{Main_Inequality_Real_axis}) for
${\rho}$, $x=0$ and $p=2m$ as follows:
$$
\left(\frac{1}{y_0(2m)}\right)^{2m-1}\le \frac{m}{\pi}.
$$
It becomes an equality in the cases $m=1$ and $m=2$ (as
$y_0(2)=\pi$ and $y_0(4)=\frac{1}{2}\sqrt[3]{4\pi}$).
\end{example}

\subsection{$(q,p)$-inequalities of Jackson-Nikolskii type for rational functions}

As a corollary of Theorem~\ref{Th_Nikolski_Real_axis_poinwise} we obtain
\begin{theorem}
Under the assumptions of Theorem~\ref{Th_Nikolski_Real_axis_poinwise}, the following
 $(q,p)$-inequality of Jackson-Nikolskii type holds:
\begin{equation}
\label{Nikolsky_ineq_Real_axis} \|R\|_{L^q(\mathbb{R})}\le
\left(\frac{m_p}{\pi}\|\mu_{*}\|_{L^\infty(\mathbb{R})}\right)^{\frac{1}{p}-\frac{1}{q}}\|R\|_{L^p(\mathbb{R})},
\qquad q>p>0,
\end{equation}
where  $m_p\in\mathbb{N}\cap[\tfrac{p}{2};1+\tfrac{p}{2})$. Moreover, if all the poles of $R$ lie outside the stripe
$$
S_\delta:=\left\{z: |\Im z|\le \delta\right\},\qquad \delta>0,
$$
then
\begin{equation}
\label{Nikolsky_ineq_Real_axis_delta} \|R\|_{L^q(\mathbb{R})}\le
\left(\frac{m_p\,n}{\pi\delta}\right)^{\frac{1}{p}-\frac{1}{q}}\|R\|_{L^p(\mathbb{R})},
\qquad q>p>0.
\end{equation}
The inequalities $(\ref{Nikolsky_ineq_Real_axis})$ and
$(\ref{Nikolsky_ineq_Real_axis_delta})$ are sharp for $(q,p)=(\infty,2)$ and $(q,p)=(\infty,4)$.
\end{theorem}
\begin{proof}
We obtain the estimate (\ref{Nikolsky_ineq_Real_axis}) from
(\ref{Main_Inequality_Real_axis}) in the same way as we obtained the estimate
(\ref{Nikolsky_ineq_circle}) from~(\ref{Main_Inequality_Circle}). To prove~(\ref{Nikolsky_ineq_Real_axis_delta}),
we use~(\ref{Nikolsky_ineq_Real_axis}) and the following estimate for the weight~(\ref{Blaschke_Product_Real_Axis}):
$$
\mu_{*}(x)=
\sum_{k=1}^{\nu}\frac{n_k\,\Im z_k}{(x-\Re z_k)^2+(\Im z_k)^2}
\le \sum_{k=1}^\nu \frac{n_k}{\Im z_k}
 \le \frac{n}{\delta},\qquad z_k\notin S_\delta,
 \qquad z_k\in \mathbb{C}^+.
 $$

The statements about sharpness follow from Examples~\ref{example3} and \ref{example4}.
\end{proof}

\subsection{Duality of the inequalities of Jackson-Nikolskii type for rational functions}

Another corollary of Theorem~\ref{Th_Nikolski_Real_axis_poinwise} is
\begin{theorem}
\label{Theorem_Alternative_Real_axis}
Under the assumptions of Theorem~\ref{Th_Nikolski_Real_axis_poinwise}, the following alternative holds for any $d\in
\mathbb{R}$:
\begin{equation}
\label{Theorem_Alternative_R}
 |R(x)|^d <\mu_{*}(x)\qquad
\text{or} \qquad |R(x)|^{2m-d}\le\frac{m}{\pi}
 \|R\|_{L^{2m}(\mathbb{R})}^{2m},\qquad m\in \mathbb{N}.
\end{equation}
In particular, for $m=d=1$ we have
$$
|R(x)|\le
\max\left\{\mu_{*}(x),\frac{1}{\pi}\,\|R\|_{L^{2}(\mathbb{R})}^{2}\right\}\qquad
\forall x\in\mathbb{R}.
$$
\end{theorem}
Indeed, if the former inequality in (\ref{Theorem_Alternative_R}) does not hold and
$\mu_{*}(x)\le|R(x)|^d$, then, substituting this estimate for $\mu_{*}$ to the left hand side of the inequality~(\ref{Main_Inequality_Real_axis}), we get the latter estimate in~(\ref{Theorem_Alternative_R}). Such alternatives enable us to obtain non-linear inequalities of Jackson-Nikolskii type for special classes of rational functions. Below we derive such inequalities for logarithmic derivatives of complex polynomials (i.e. for simple partial fractions).

\subsection{Inequalities of Jackson-Nikolskii type for simple partial fractions}
\label{SPF_Section}

Recall that in rational approximation theory a \textit{simple partial fraction} (\textit{SPF})
of degree $n$ is a function of the form
\begin{equation*}
\label{rho} \rho_n(z):=\sum_{k=1}^n\frac{1}{z-z_k},\qquad z_k\in
\mathbb{C}.
\end{equation*}
Obviously, SPFs are the logarithmic derivatives of algebraic polynomials with the zeros $z_k$.

Problems connected with various extremal properties of SPFs attract attention of many authors  (see the survey  \cite{DanKomChu} and the references there and in the current section). One of such problems is the inequalities between various metrics (i.e. the inequalities of Jackson-Nikolskii type). They were first considered in~\cite{D1994}. Later on, various refinements and extensions of the results from \cite{D1994} were obtained in   \cite{DAN-DOD,Kayumov2012,Kayumov2011,Kayumov2012mz} and other papers.

The purpose of this section is to further develop this topic applying the quadrature formulas. The specificity of SPFs allows us to use the above-obtained inequalities for general rational functions effectively in the case of SPFs.  In particular, we will derive sharp inequalities for SPFs on the real axis. Let us formulate the main result.

If all the poles $z_k$ of the function $\rho_n$ lie on the upper (lower)
half-plane $\mathbb{C}^+$ ($\mathbb{C}^-$), then we write $\rho_n(z)=\rho^+_n(z)$ ($\rho_n(z)=\rho^-_n(z)$) to make things clear. Set
$$
d(\rho^{\pm}_n;p):=2\pi\;\frac{\|\rho_n^{\pm}\|^{p-1}_{L^\infty(\mathbb{R})}}
{\|\rho_n^{\pm}\|_{L^{p}(\mathbb{R})}^{p}},\qquad p>1.
$$
\begin{theorem}
\label{SPF_Nik1+++}
It holds that
\begin{equation}
\label{SPF_Nik1} d(\rho^{\pm}_n;p)\le 2m_p,\qquad
m_p\in \mathbb{N}\cap [\tfrac{p}{2},1+\tfrac{p}{2}),\qquad p>1;
\end{equation}
\begin{equation}
\label{SPF_Nik2} d(\rho^{\pm}_n;p)\ge \frac{1}{n}\,
 {\cos \pi(1-\tfrac{p}{2})}, \quad 1<p\le 2,
 \qquad d(\rho^{\pm}_n;p) \ge
\frac{1}{n},\quad p\ge 2.
\end{equation}
The inequality $(\ref{SPF_Nik1})$ is sharp for $p=2$ and $p=4$.
\end{theorem}
It is worth mentioning that the power of $n$ in the
estimate~(\ref{SPF_Nik2}) is exact; one can find the corresponding
examples in~\cite{DAN-DOD}.

\begin{proof}
It is enough to prove the inequalities only for $d(\rho^{+}_n;p)$. Note that
 $\mu_{*}(x)=\Im \rho_n^+(x)\le |\rho_n^+(x)|$ and choose $x$
so that $\rho_n^+(x)=\|\rho_n^+\|_{L^\infty(\mathbb{R})}$. Then the estimate~(\ref{SPF_Nik1}) follows from the alternative~(\ref{Theorem_Alternative_R}) with $d=1$ and $1<p\le 2m$ (one has to use the same trick as in Section~\ref{L-p-q+++}):
$$
\|\rho_n^+\|_{L^{\infty}(\mathbb{R})}^{2m-1}\le \frac{m}{\pi}\;
 \|\rho_n^+\|_{L^{2m}(\mathbb{R})}^{2m}
\le \frac{m}{\pi}
 \|\rho_n^+\|^{2m-p}_{L^{\infty}(\mathbb{R})}
\|\rho_n^+\|^{p}_{L^{p}(\mathbb{R})},
$$
where it is better to take the minimal possible value of  $m=m_p$ from
(\ref{SPF_Nik1}). We are now in the position to obtain the inequality~(\ref{SPF_Nik2}).
To begin with, take $1<p\le 2$. We need \cite[Inequality~(18)]{D2010}:
\begin{equation}
\label{Dan2010_1} \|\rho_n^+\|^p_{L^p(\mathbb{R})}
\cos\pi(1-\tfrac{p}{2})\le 2\pi\;\Im\left(e^{-i\pi
(1-\tfrac{p}{2})}
\sum_{k=1}^n(\rho_n^+(\overline{z_k}))^{p-1}\right),\qquad 1<p<3.
\end{equation}
By the maximum modulus principle, the right hand side of the inequality~(\ref{Dan2010_1}) is bounded from above by $2\pi n
\|\rho_n^+\|_{L^\infty(\mathbb{R})}^{p-1}$. This immediately implies the former inequality in~(\ref{SPF_Nik2}) for $1<p\le 2$. A similar inequality can be obtained with the help of (\ref{Dan2010_1}) even for $2\le p<3$ but the latter inequality in~(\ref{SPF_Nik2}) is obviously more precise in this case.

Now suppose that $p\ge 2$. Using the identity~(\ref{Int_formula_Real_axis_DanSem++}) with $x_k=x_k(\varphi)$ and some $\varphi$ twice, we get
$$
\|\rho_n^+\|_{L^{2m}(\mathbb{R})}^{2m}=
\frac{\pi}{2m}\sum_{k=1}^{2mn}\frac{|\rho_n^+(x_k)|^{2m}}{\Im
\rho_n^+(x_k)}\le \|\rho_n^+\|^{2m-2}_{L^{\infty}(\mathbb{R})}
 \frac{\pi}{2m}\;\sum_{k=1}^{2mn}\frac{|\rho_n^+(x_k)|^2}{\Im
\rho_n^+(x_k)}=
 \|\rho_n^+\|^{2m-2}_{L^{\infty}(\mathbb{R})}
\|\rho_n^+\|^2_{L^{2}(\mathbb{R})}.
$$
This and the inequality $\|\rho_n^+\|_{L^{2}(\mathbb{R})}^2
 \le 2\pi n \|\rho_n^+\|_{L^{\infty}(\mathbb{R})}$
from \cite[Formula (36)]{D2010} yield
$$
\|\rho_n^+\|_{L^{2m}(\mathbb{R})}^{2m}\le 2\pi
n\;\|\rho_n^+\|^{2m-1}_{L^{\infty}(\mathbb{R})}\qquad
\Rightarrow\qquad
A:=2\pi
n\; \|\rho_n^+\|^{2m-1}_{L^{\infty}(\mathbb{R})}
  \|\rho_n^+\|_{L^{2m}(\mathbb{R})}^{-2m}\ge 1.
$$
Furthermore, if $m\in
\mathbb{N}$ is such that $2m\le p<2(m+1)$, then
$$
\|\rho_n^+\|_{L^{p}(\mathbb{R})}^{p}=
\int_\mathbb{R}|\rho_n^+|^{p-2m}|\rho_n^+|^{2m}
\le \|\rho_n^+\|^{p-2m}_{L^{\infty}(\mathbb{R})}
\int_\mathbb{R}|\rho_n^+|^{2m}=
\|\rho_n^+\|^{p-2m}_{L^{\infty}(\mathbb{R})}
\|\rho_n^+\|_{L^{2m}(\mathbb{R})}^{2m},
$$
hence
$$
\|\rho_n^+\|_{L^{p}(\mathbb{R})}^{p}\le A\,
\|\rho_n^+\|_{L^{p}(\mathbb{R})}^{p}\le
A\,\|\rho_n^+\|^{p-2m}_{L^{\infty}(\mathbb{R})}
\|\rho_n^+\|_{L^{2m}(\mathbb{R})}^{2m}=2\pi
n\;\|\rho_n^+\|^{p-1}_{L^{\infty}(\mathbb{R})},
$$
which is equivalent to the latter inequality in~(\ref{SPF_Nik2}).

We now prove the sharpness of the estimate~$(\ref{SPF_Nik1})$. For
$p=2,4$ (i.e. for $m_2=1$ and $m_4=2$) this estimate reads as
$d(\rho^{+}_n;2)\le 2$ and $d(\rho^{+}_n;4)\le 4$, correspondingly. In both cases it becomes an equality for the single-pole SPF $\rho(p;z)$ from~(\ref{rho_wave}). This follows from the obvious identities
$$
\|{\rho(p,\cdot)}\|_{L^p(\mathbb{R})}=1, \qquad
\|{\rho(2,\cdot)}\|_{L^{\infty}(\mathbb{R})}=1/\pi,\qquad
\|{\rho(4,\cdot)}\|_{L^{\infty}(\mathbb{R})}=(2/\pi)^{1/3}.
$$
\end{proof}
\begin{remark} Let us compare Theorem~\ref{SPF_Nik1+++} with several previous results. It is convenient here to introduce the quantity
$$
D(\rho^{\pm}_n;p):=
\left(\frac{1}{2\pi}d(\rho^{\pm}_n;p)\right)^{{p'}/p}=\frac{\|\rho_n^{\pm}\|_{L^\infty(\mathbb{R})}}
{\|\rho_n^{\pm}\|_{L^{p}(\mathbb{R})}^{{p'}}},\qquad
 \frac{1}{{p'}}+\frac{1}{p}=1,
 $$
so that the numerator and the denominator of $D(\rho^{\pm}_n;p)$ have the same dimension as the SPF itself. In particular, under the transformation $\tilde\rho^{\pm}_n(x):=a\rho^{\pm}_n(ax)$, $a>0$,
preserving the form of SPF, the norms that we estimate depend on $a$ linearly:
$$
\|\tilde\rho^{\pm}_n\|_{L^p(\mathbb{R})}^{p'} =a
\|\rho^{\pm}_n\|_{L^p(\mathbb{R})}^{p'}.
$$
In \cite[Corollary]{D1994} for $p>1$ the following  estimate (being independent of $n$) is proved:
\begin{equation}
\label{dan-1994} D(\rho^{\pm}_n;p)\le \sigma(p):=p\,\sin^{-p'}
\frac{\pi}{p}.
\end{equation}
Note that $\lim_{p\to \infty}\sigma(p)=\infty$ and $\lim_{p\to
1+}\sigma(p)=\infty$, so the estimate loses meaning if $p$ is large enough or close to~$1$. Some majorant, being more precise than in (\ref{dan-1994}), can be derived from~\cite{Borodin2007} but it also tends to infinity as $p\to 1+$ or $p\to\infty$. On the contrary, (\ref{SPF_Nik1}) and
(\ref{SPF_Nik2}) give
\begin{equation*}
\label{left-right+++} \sigma_1(p):=\left(\frac{1}{2\pi
n}\right)^{\frac{1}{p-1}}\le D(\rho^{\pm}_n;p)\le \sigma_2(p):=
 \left(\frac{p+2}{2\pi}\right)^{\frac{1}{p-1}}< 1.1,\qquad p\ge 2,
\end{equation*}
where $\lim_{p\to \infty}\sigma_{1,2}(p)=1$ for a fixed $n$,
i.e. the inequalities coincide (they coincide even if $p\to
1+$ as then the minorant and the majorant tend to zero, see the former inequality in~(\ref{SPF_Nik2})). In this sense the inequality~(\ref{SPF_Nik1}) refines the estimate (\ref{dan-1994}) and the one from~\cite{Borodin2007} significantly.
\end{remark}
\begin{theorem}
\label{SPF_Nik1_common+++} Without any assumptions on the poles, it holds that
\begin{equation}
\label{SPF_Nik2_common+++} \|\rho_n\|_{L^q(\mathbb{R})}^{q'} \le
2^{\,q'-p'} \left(\frac{m_p}{\pi}\right)^{{p'}{q'}
\left(\frac{1}{p}-\frac{1}{q}\right)}(1+h_p)^{p'}
\|\rho_n\|_{L^p(\mathbb{R})}^{p'}, \quad
\frac{1}{q}+\frac{1}{{q'}}=1,\quad \frac{1}{p}+\frac{1}{{p'}}=1,
\end{equation}
where $1<p<q\le \infty$, $m_p\in \mathbb{N}\cap
[\tfrac{p}{2},1+\tfrac{p}{2})$, and $h_p$ is the norm of the Hilbert transform. Recall that
$$
h_p=\left\{\begin{array}{ll}
      \tan \tfrac{\pi}{2p}, & 1<p\le 2, \\
      \cot  \tfrac{\pi}{2p} & 2\le p<\infty.
    \end{array}\right.
$$
\end{theorem}
\begin{proof}
To get an intermediate estimate (see (\ref{dan-DOD+++})),
we repeat the arguments from \cite[proof of Theorem~2]{DAN-DOD}.
We suppose that $\alpha=\|\rho^{+}_n\|_{L^\infty(\mathbb{R})}$ and introduce the new SPF $\hat{\rho}(x)=\alpha^{-1}\rho^{+}_n\left(x\alpha^{-1}\right)$ with
the $\sup$-norm on ${\mathbb R}$ equal to $1$. It can be easily checked (see \cite{DAN-DOD}) that
$$
\|\hat{\rho}\|_{L^{q}(\mathbb{R})}^q=
 {\alpha}^{1-q}\|\rho^{+}_n\|_{L^{q}(\mathbb{R})}^q,\quad
\|\hat{\rho}\|_{L^{p}(\mathbb{R})}^p=
{\alpha}^{1-p}\|\rho^{+}_n\|_{L^{p}(\mathbb{R})}^p.
$$
Since $\|\hat{\rho}\|_{L^\infty(\mathbb{R})}=1$, we have $\|\hat{\rho}\|^q_q \le
\|\hat{\rho}\|^p_p$ for $q>p$, i.e.
$$
{\alpha}^{1-q}\|\rho^{+}_n\|^q_{L^{q}(\mathbb{R})} \le
{\alpha}^{1-p}\|\rho^{+}_n\|^p_{L^{p}(\mathbb{R})},\qquad
\|\rho^{+}_n\|^q_{L^{q}(\mathbb{R})} \le
{\alpha}^{q-p}\|\rho^{+}_n\|^p_{L^{p}(\mathbb{R})}.
$$
Taking into account the estimate (\ref{SPF_Nik1}) for $\alpha$, we conclude that
$$
\|\rho^{+}_n\|^q_{L^{q}(\mathbb{R})} \le
\left(\frac{m_p}{\pi}\right)^{(q-p)/(p-1)}
 \|\rho^{+}_n\|_{L^{p}(\mathbb{R})}^{p'(q-p)}\cdot
\|\rho^{+}_n\|^p_{L^{p}(\mathbb{R})}=\left(\frac{m_p}{\pi}\right)^{(q-p)/(p-1)}
 \|\rho^{+}_n\|_{L^{p}(\mathbb{R})}^{p'(q-1)}.
$$
The same estimate holds for $\rho^{-}_n$. Raising each side to the power  $1/(q-1)$ gives
\begin{equation}
\label{dan-DOD+++} \|\rho_n^{\pm}\|_{L^q(\mathbb{R})}^{q'}\le
\left(\frac{m_p}{\pi}\right)^{{p'}{q'}\left(\frac{1}{p}-
\frac{1}{q}\right)}\|\rho_n^{\pm}\|_{L^p(\mathbb{R})}^{p'}.
\end{equation}
Furthermore, it follows from \cite[Inequality (21)]{D2010} that
\begin{equation}
\label{ineq_D2010} \|\rho_n^{\pm}\|_{L^p(\mathbb{R})}\le
\tfrac{1}{2}(1+h_p) \|\rho_n\|_{L^p(\mathbb{R})},\qquad
\rho_n=\rho_n^{+}+\rho_n^{-}.
\end{equation}
From this, by the Minkowski and Jensen inequalities and (\ref{ineq_D2010}), we obtain the required result:
\begin{eqnarray*}
\|\rho_n\|_{L^q(\mathbb{R})}^{q'} &\le & \left(\|\rho_n^{+}\|_{L^q(\mathbb{R})}+\|\rho_n^{-}\|_{L^q(\mathbb{R})}\right)^{q'}\\
&\le & 2^{\,q'-1} \left(\|\rho_n^{+}\|_{L^q(\mathbb{R})}^{q'}+\|\rho_n^{-}\|_{L^q(\mathbb{R})}^{q'}\right)\\
 &\le & 2^{\,q'-1} \left(\frac{m_p}{\pi}\right)^{{p'}{q'}\left(\frac{1}{p}-
\frac{1}{q}\right)}\left(\|\rho_n^{+}\|_{L^p(\mathbb{R})}^{p'}+\|\rho_n^{-}\|_{L^p(\mathbb{R})}^{p'}\right)\\
 &\le& 2^{\,q'-p'} \left(\frac{m_p}{\pi}\right)^{{p'}{q'}\left(\frac{1}{p}-
\frac{1}{q}\right)}(1+h_p)^{p'} \|\rho_n\|_{L^p(\mathbb{R})}^{p'}.
\end{eqnarray*}
\end{proof}

Note that the constant in (\ref{SPF_Nik2_common+++}) is more precise than the analogous one in \cite[Theorem~2]{DAN-DOD}.

\section{Quadrature formulas and inequalities in the case of the semiaxis~$\mathbb{R}^+$}

\subsection{Inequalities of Jackson-Nikolskii type for rational functions}
As in the cases of the circle and the real axis, the quadrature formula (\ref{Quadrature_Semiaxis}) and the trick in Section~\ref{L-p-q+++} yield
\begin{theorem}
\label{Th_Nikolski_Real_Semiaxis_poinwise} Let $R(z)$ be a proper rational fraction of degree~$n$ whose poles do not belong to the real semiaxis~$\mathbb{R}^+$. Then it holds that
\begin{equation}
\label{Main_Inequality_Real_Semiaxis}
 \frac{|R(x^2)|^{2m}}{\mu_{*}(x)}\le
\frac{m}{\pi}\|R\|_{L^{2m}(\mathbb{R}^+;\frac{1}{\sqrt{x}})}^{2m}\qquad
\forall x\in \mathbb{R},\qquad m\in \mathbb{N};
\end{equation}
\begin{equation*}
\label{Main_Inequality_Real_Semiaxis+++}
\|R\|_{L^q(\mathbb{R}^+;\frac{1}{\sqrt{x}})}\le
\left(\frac{m_p}{\pi}
\|\mu_*\|_{L^\infty(\mathbb{R}^+)}\right)^{\frac{1}{p}-\frac{1}{q}}
\|R\|_{L^p(\mathbb{R}^+;\frac{1}{\sqrt{x}})},\qquad m_p\in
\mathbb{N}\cap [\tfrac{p}{2};1+\tfrac{p}{2}),
\end{equation*}
where the $($even$)$ function $\mu_{*}$ is defined in Theorem~$\ref{Th_Real_semiaxis}$ and $0<p< q\le\infty$.

Moreover, if all the poles of $R$ lie in the exterior of the parabola
$$
x=\frac{1}{4}\frac{y^2}{\delta^2}-\delta^2, \qquad \delta>0,
$$
then for $0<p< q\le\infty$,
\begin{equation*}
\label{Main_Inequality_Real_Semiaxis_Delta}
\|R\|_{L^q(\mathbb{R}^+;\frac{1}{\sqrt{x}})}\le
\left(\frac{2m_p\,n}{\pi\delta} \right)^{\frac{1}{p}-\frac{1}{q}}
\|R\|_{L^p(\mathbb{R}^+;\frac{1}{\sqrt{x}})}.
\end{equation*}
\end{theorem}
The latter statement follows from the fact that under the
substitution $z=\tilde{z}^2$ the exterior of the stripe $|{\rm Im}\,
\tilde{z}|\ge \delta$ turns into the exterior of the above-mentioned
parabola. Therefore $|{\rm Im}\; \sqrt{z_k}|\ge \delta$ and
$\|\mu_*\|_{L^\infty(\mathbb{R}^+)}\le 2n/\delta$ in
Theorem~\ref{Th_Real_semiaxis}.

\subsection{Inequalities of Jackson-Nikolskii type for simple partial fractions}

Let all the poles of the SPF $\rho_n(x)=\sum_{k=1}^n(x-z_k)^{-1}$ lie in the acute beam angle $2\alpha<\pi$ with the bisector on the negative semiaxes:
$$
z_k=r_ke^{i\varphi_k}, \qquad \varphi_k\in
(\pi-\alpha,\pi+\alpha),\qquad \alpha\in (0,\pi/2).
$$
\begin{theorem}
\label{D+NIKOLSKY_R+} Given $m\in \mathbb{N}$, it holds that
\begin{equation}
\label{1000+++} \|\rho_n\|_{L^\infty(\mathbb{R}^+)}^{2m-1/2}\le
\left(\sum_{k=1}^n\frac{1}{r_k}\right)^{2m-1/2}\le \frac{2
\sqrt{n}}{\cos^{2m}
\alpha}\frac{m}{\pi}\|\rho_n\|_{L^{2m}(\mathbb{R}^+;\frac{1}{\sqrt{x}})}^{2m}.
\end{equation}
\end{theorem}
\begin{proof} For $x\in \mathbb R^{+}$ we have
$$
S:=\sum_{k=1}^n\frac{1}{r_k}\ge
|\rho_n(x)|\ge\sum_{k=1}^n\frac{|\cos \arg (x-z_k)|}
{|x-z_k|}\ge\sum_{k=1}^n\frac{\cos \alpha} {|x-z_k|}.
$$
This gives $|\rho_n(0)|\ge S\cos \alpha$ for $x=0$. Consequently, taking into account the inequalities $\mu_*(0)\le 2\,\sum_k r_k^{-1/2}$ and (\ref{Main_Inequality_Real_Semiaxis}), we obtain
$$
\frac{\cos^{2m} \alpha}{2}\, S^{2m}\left(\sum_{k=1}^n\frac{1}
{\sqrt{r_k}}\right)^{-1}\le \frac{|\rho_n(0)|^{2m}}{\mu_*(0)} \le
\frac{m}{\pi}\|\rho_n\|_{L^{2m}(\mathbb{R}^+;\frac{1}{\sqrt{x}})}^{2m}.
$$
Note that for each fixed $S$ the minimum in the left hand side is reached when $r_k=n/S$:
$$
S^{2m}\left(\sum_{k=1}^n\frac{1} {\sqrt{r_k}}\right)^{-1}\ge
\frac{1}{\sqrt{n}}S^{2m-1/2}\ge \frac{1}{\sqrt{n}}
\|\rho_n\|_{L^\infty(\mathbb{R}^+)}^{2m-1/2},
$$
which leads to (\ref{1000+++}) if we take into account the previous inequality.
\end{proof}

\section{Quadrature formulas and inequalities in the case of the segment $[-1,1]$}
\subsection{Inequalities of Jackson-Nikolskii type for rational functions}
The following result follows from the quadrature formula (\ref{Segment_L_m})
via calculations analogous to those in Section~\ref{L-p-q+++}.
\begin{theorem}
\label{Segment_Th_Real_axis++} Let $R(z)$ be a rational function of degree~$n$ whose poles do not belong to the segment~$I$. Then
\begin{equation}
\label{Segment_Main_Inequality}
 \frac{|R(x)|^{2m}}{m\mu_0(\zeta)+1}\le
\frac{1}{\pi}\|R\|_{L^{2m}(I;\,\omega)}^{2m},\qquad
x=\tfrac{1}{2}\left(\zeta+1/\zeta\right)\in I\qquad \forall
\zeta\in \gamma_1,\qquad m\in
\mathbb{N};
\end{equation}
\begin{equation}
\label{Main_Inequality_Segment+++} \|R\|_{L^q(I;\,\omega)}\le
\left(\frac{m_p\,\|\mu_0\|_{L^\infty(\gamma_1)}+1}{\pi}
\right)^{\frac{1}{p}-\frac{1}{q}} \|R\|_{L^{p}(I;\,\omega)},\qquad m_p\in
\mathbb{N}\cap [\tfrac{p}{2};1+\tfrac{p}{2}),
\end{equation}
where $0<p< q\le\infty$. Moreover, if all the poles of $R$ lie in the exterior of the ellipse
$$
\frac{1}{2}\left({\delta}+\frac{1}{{\delta}}\right)\cos t+
\frac{i}{2}\left({\delta}-\frac{1}{{\delta}}\right)\sin t,\qquad
\delta>1,\qquad t\in{\mathbb R},
$$
then for $0<p< q\le\infty$ we have
\begin{equation}
\label{ELLIPS+++} \|R\|_{L^q(I;\,\omega)}\le
\left(\frac{1}{\pi}\right)^{\frac{1}{p}-\frac{1}{q}}
\left(2m_p\,n\cdot\frac{\delta+1}{\delta-1}+1\right)^{\frac{1}{p}
-\frac{1}{q}} \|R\|_{L^{p}(I;\,\omega)}.
\end{equation}
\end{theorem}

\subsection{Inequalities of Jackson-Nikolskii type for polynomials}
\label{Segment_SPF_Section}

Now we use the parametric quadrature formulas to estimate the norms of complex polynomials on the segment~$I$. If we put $R=P_n$ in Theorem~\ref{Segment_Th_Real_axis}, where $P_n$ is a polynomial
of degree $n$ with complex coefficients, then the function
$R_1(z)=P_n(\tfrac{1}{2}(z+1/z))$ of degree $2n$ has poles only at $z=0$ and $z=\infty$, each one of multiplicity~$n$.
Consequently, the function
 ${\mathcal R}$ has only one pole~$z=0$ of multiplicity~$2n$ in the unit disc. By this reason it follows from the inequality (\ref{Main_Inequality_Segment+++}) and the identity $|\mu_0(\zeta)|_{\gamma_1}\equiv 2 n$ that
\begin{equation}
\label{Segment_Polynomials} \|P_n\|_{L^q(I;\,\omega)}\le
\left(\frac{2m_p\,n+1}{\pi} \right)^{\frac{1}{p}-\frac{1}{q}}
\|P_n\|_{L^{p}(I;\,\omega)},\qquad p< q\le\infty.
\end{equation}
This estimate is sharp; for $(q,p)=(\infty,2)$ it becomes an equality for the polynomials
\begin{equation}
\label{EXTREMA_POLY}
P^*_n(x)=P^*_n(\tfrac{1}{2}(z+1/z))=
C\frac{1}{z^n}\frac{z^{2n+1}-1}{z-1}.
\end{equation}
These polynomials exist and are actually  the Jacobi polynomials
$P^*_n(x)=P^{(1/2,-1/2)}_n(x)$ of degree $n$, which satisfy the equation
$$
(1-x^2)y''(x)-(1+2x)y'(x)+n(n+1)y(x)=0.
$$
Indeed, it can be easily checked that the substitution $Y(z)/(z-1)=y(x)$ gives the simple equation
$$
Y''(z)z^2-n(n+1)Y(z)=0
$$
with the general solution $Y(z)=c_1z^{n+1}+c_2z^{-n}$. For the polynomial solution $P^*_n(x)=Y(z)/(z-1)$ the identity $c_1=-c_2$ is necessary so that (\ref{EXTREMA_POLY}) holds. It is now sufficient to put $m=1$ and $\varphi=0$ in (\ref{Segment_L_m}) and take $x_k=x_k(1,0)$.
The notches $\zeta_k(1,0)$ are the roots of the equation $z^{2n+1}=1$ in this case. Therefore all the summands but one in
(\ref{Segment_L_m}) vanish and hence the inequality~(\ref{Segment_Polynomials}) for $(q,p)=(\infty,2)$ ($m_2=1$)
becomes an equality. Furthermore,
$$
|P^*_n(1)|=2n+1= \sqrt{\frac{2n+1}{\pi}}
\|P^*_n\|_{L^{2}(I;\,\omega)}.
$$
Inequalities of the form (\ref{Segment_Polynomials}) with various constants are known, see, for example, \cite{Arestov-D,Arestov-Deykalova,Lupas,
{DanSem2016}}. The paper \cite{Arestov-Deykalova} contains, among other results, the inequality for
$(q,p)=(\infty,2)$ with the same sharp constant as in~(\ref{Segment_Polynomials}).

\section*{Acknowledgments}

P. Chunaev is supported by the Russian Foundation for Basic Research (project 16-31-00252 mol\_a) and by the European Research Council (ERC grant 320501, FP7/2007-2013). V. Danchenko is supported by the Russian Ministry of Education and Science  (task 1.574.2016/1.4).
The authors are also partially supported by the Russian Foundation for Basic Research (project 18-01-00744 A).

The authors are grateful to A. Baranov, R. Zarouf and S. Tikhonov for valuable discussions.

\end{document}